\documentclass[12pt]{amsart}
\usepackage{amscd, amsfonts, amssymb, amsthm}
\usepackage{parskip, fullpage, verbatim}
\usepackage[colorlinks, breaklinks, linkcolor=blue]{hyperref} 
\usepackage{breakurl}
\usepackage{array}
\usepackage{booktabs}
\usepackage{wasysym} 

\newcommand\CA{{\mathcal A}}

\newcommand\CIF{{\mathcal {IF}}}

\newcommand\BBC{{\mathbb C}}
\newcommand\BBN{{\mathbb N}}

\newcommand\BBZ{{\mathbb Z}}

\newcommand {\GAP}{\textsf{GAP}}  
\newcommand {\CHEVIE}{\textsf{CHEVIE}}

\newcommand\codim{\operatorname{codim}}

\newcommand\Der{{\operatorname{Der}}}

\newcommand\Fix{{\operatorname{Fix}}}

\newcommand\GL{\operatorname{GL}}

\newcommand\pdeg{\operatorname{pdeg}}

\newcommand\inverse{^{-1}}

\numberwithin{equation}{section}

\theoremstyle{plain}
\newtheorem{lemma}[equation]{Lemma}
\newtheorem{theorem}[equation]{Theorem}

\newtheorem{proposition}[equation]{Proposition}
\theoremstyle{definition}
\newtheorem{defn}[equation]{Definition}
\newtheorem{remark}[equation]{Remark}

\subjclass[2010]{Primary 20F55, 52B30, 52C35, 14N20; Secondary 13N15}

\begin{document}

\title[On supersolvable reflection arrangements]
{On supersolvable reflection arrangements}


\author[T. Hoge]{Torsten Hoge}
\address
{Fakult\"at f\"ur Mathematik,
Ruhr-Universit\"at Bochum,
D-44780 Bochum, Germany}
\email{torsten.hoge@rub.de}

\author[G. R\"ohrle]{Gerhard R\"ohrle}
\address
{Fakult\"at f\"ur Mathematik,
Ruhr-Universit\"at Bochum,
D-44780 Bochum, Germany}
\email{gerhard.roehrle@rub.de}

\keywords{Complex reflection groups,
reflection arrangements, free arrangements, 
supersolvable arrangements}

\allowdisplaybreaks

\begin{abstract}
Let $\CA = (\CA,V)$ be a complex hyperplane arrangement and let
$L(\CA)$ denote its intersection lattice.
The arrangement $\CA$ is called supersolvable, provided 
its lattice $L(\CA)$ is supersolvable,
a notion due to 
Stanley \cite{stanley:super}. 
Jambu and Terao \cite[Thm.\ 4.2]{jambuterao:free}
showed that every 
supersolvable arrangement is inductively free, 
a notion due to Terao, \cite{terao:freeI}.
So this is a natural subclass of this particular class of free 
arrangements. 

Suppose that $W$ is a finite, unitary reflection group acting on the complex 
vector space $V$. Let $\CA = (\CA(W), V)$ be the associated 
hyperplane arrangement of $W$.
In \cite{hogeroehrle:indfree}, we determined all 
inductively free reflection arrangements.

The aim of this note is to 
classify all supersolvable reflection arrangements.
Moreover, we characterize 
the irreducible arrangements in this class 
by the presence of modular elements of rank $2$ in their intersection lattice.
\end{abstract}

\maketitle


\section{Introduction}

Let $\CA  = (\CA,V)$ be a complex hyperplane arrangement and let
$L(\CA)$ denote its intersection lattice.
We say that $\CA$ is \emph{supersolvable}, provided $L(\CA)$
is supersolvable, see Definition \ref{def:super}.
Jambu and Terao \cite[Thm.\ 4.2]{jambuterao:free}
have shown that every 
supersolvable arrangement is \emph{inductively free},
a notion due to Terao, \cite{terao:freeI};
see Definition \ref{def:indfree}.

Now suppose that $W$ is a finite, unitary 
reflection group acting on the complex 
vector space $V$.
Let $\CA = (\CA(W),V)$ be the associated 
hyperplane arrangement of $W$.
Terao \cite{terao:freeI} has shown that each reflection
arrangement $\CA$ is free and that 
the multiset of exponents 
$\exp \CA$ of $\CA$ is given by the 
coexponents 
of $W$; cf.\ \cite[\S 6]{orlikterao:arrangements}.
There is the stronger notion of 
inductive freeness referred to above, 
cf.\ Definition \ref{def:indfree}.
In \cite{hogeroehrle:indfree}, we classified all 
inductively free reflection arrangements,
cf.\ Theorem \ref{thm:ind-free}.
In this note we determine the subclass of all 
supersolvable reflection arrangements:
apart from the braid and rank $2$ arrangements, 
these consist only of the reflection arrangements
of the groups $G(r,p,\ell)$ for $r,\ell \ge 2$ and $p \ne r$,
cf.\   Theorem \ref{super}.
In addition, the irreducible 
arrangements of this nature
are characterized 
by the presence of a modular rank $2$ element 
in the intersection lattice,
see  Theorem \ref{super-rank2}.

Within the theory of hyperplane arrangements
supersolvability is rather strong condition, 
as it implies essentially every desirable 
property, such as factored, free, 
fiber-type, $K(\pi,1)$,
rational $K(\pi,1)$,  etc.\
see \cite[(2.8)]{falkRandell:fiber-type} for details.
We  briefly discuss three relevant properties.

Firstly, the Poincar\'e polynomial $\pi(\CA,t)$ 
of the lattice $L(\CA)$
of a supersolvable arrangement $\CA$ 
factors into linear terms over 
$\BBZ[t]$,
thanks to a theorem of Stanley \cite{stanley:super},
see also 
\cite[Thm.\ 2.63]{orlikterao:arrangements}.
Precisely, 
$\pi(\CA,t)$ 
factors into linear terms as follows:
\[
\pi(\CA,t) = \prod_{i=1}^\ell (1 + b_i t),
\]
where the coefficients $b_i$ are positive integers.

Secondly, as mentioned above,  
a supersolvable arrangement $\CA$ is always 
inductively free, 
\cite[Thm.\ 4.2]{jambuterao:free},
see also \cite[Thm.\ 4.85]{orlikterao:arrangements}.
So we can talk about the exponents $\exp \CA$ 
of $\CA$ in this case, cf.\ \S \ref{ssect:refl}.
In the factorization of $\pi(\CA,t)$ above,
the occurring coefficients $b_i$
are precisely the exponents of $\CA$,
\cite[Thm.\ 4.137]{orlikterao:arrangements}.

Thirdly, if $\CA$ is supersolvable, then 
$\CA$ is a $K(\pi,1)$-arrangement. 
For, Falk and Randell
\cite{falkRandell:fiber-type}
proved that 
\emph{fiber-type} arrangements are
always $K(\pi,1)$
(cf.\ \cite[Prop.\ 5.12]{orlikterao:arrangements})
and by work of Terao \cite{terao:modular}, 
the classes of
fiber-type arrangements and of
supersolvable arrangements coincide
(cf.\ \cite[Thm.\ 5.112]{orlikterao:arrangements}).
Therefore, every supersolvable arrangement is 
$K(\pi, 1)$.

Since a supersolvable 
arrangement is  inductively free, 
we only need to consider the latter
for our classification. 
We recall the 
main result from \cite{hogeroehrle:indfree}.

\begin{theorem}
\label{thm:ind-free}
For $W$ a finite complex reflection group,  
the reflection arrangement $\CA(W)$ of $W$ is 
inductively free if and only if 
$W$ does not admit an irreducible factor
isomorphic to a monomial group 
$G(r,r,\ell)$ for $r, \ell \ge 3$, 
$G_{24}, G_{27}, G_{29},G_{31},  G_{33}$, or $G_{34}$.
\end{theorem}

Jambu and Terao have already observed
that $\CA(D_4)$ is not supersolvable, \cite[Ex.\ 5.5]{jambuterao:free}.
Thus,  supersolvable reflection arrangements 
form a proper subclass of the class of inductively free reflection arrangements.

It is straightforward to see that any rank $2$ 
arrangement is supersolvable, cf.\ Remark \ref{rem:2-arr}. 
In 1962, Fadell and Neuwirth \cite{fadellneuwirth}
showed that the
braid arrangement 
is fiber-type and 
in 1973 Brieskorn \cite{brieskorn:tresses} 
proved this for the reflection arrangement of the
hyperoctahedral group. 
In 1983, using Brieskorn's iterated fibration method,
Orlik and Solomon \cite[\S 4]{orliksolomon:coxeter} observed 
that indeed every $\CA(G(r,p,\ell))$ is fiber-type,
for $r, \ell \ge 2$  and $p \ne r$.
So $\CA$ is supersolvable in each of these cases.
Our classification asserts that 
this list is actually  complete.

\begin{theorem}
\label{super}
For $W$ a finite
complex reflection group,  
$\CA(W)$ is 
supersolvable if and only if 
any irreducible factor of $W$ is
of rank at most $2$, 
is isomorphic either to 
a Coxeter group 
of type $A_\ell$ or $B_\ell$ for $\ell \ge 3$, or to a  
monomial group $G(r,p,\ell)$
for $r, \ell \ge 3$  and $p \ne r$.
\end{theorem}

In 1962,  Fadell and Neuwirth \cite{fadellneuwirth}
proved that the braid arrangement $\CA(A_\ell)$
is $K(\pi,1)$. In 1973, 
Brieskorn \cite{brieskorn:tresses} extended this 
result to a large class of 
Coxeter groups and 
conjectured that this is the case for 
every Coxeter group.
This was proved subsequently by Deligne \cite{deligne},
who showed that  every 
\emph{simplicial} arrangement is $K(\pi,1)$.
All reflection arrangements $\CA(W)$ 
have been known to be $K(\pi,1)$ since the late 1980s, 
with the exception of the ones for
the six exceptional groups listed in Theorem \ref{thm:ind-free}, see
\cite[\S 6.6]{orlikterao:arrangements}.
Since these arrangements are not inductively free,
they are not supersolvable (or fiber-type) and so 
one cannot easily deduce 
that they are also $K(\pi,1)$.
These outstanding cases were settled only 
recently 
by Bessis \cite{bessis:kpione}.

The definition of 
supersolvability of $\CA$ entails the existence of 
modular elements in $L(\CA)$ of any possible rank;
see \S \ref{ssect:supersolve} for the notion of 
modular elements.
Strikingly, our second main result shows that 
irreducible,
supersolvable reflection arrangements 
are characterized merely by the presence
of a modular element of  rank $2$.

\begin{theorem}
\label{super-rank2}
For $W$ a finite, irreducible
complex reflection group of rank at least $2$,  
$\CA(W)$ is 
supersolvable if and only if 
there exists a modular element of rank $2$
in its lattice $L(\CA(W))$.
\end{theorem}

The condition of irreducibility in  
Theorem \ref{super-rank2} 
is necessary, see Remark \ref{rem:reducible}.

The paper is organized as follows.
In \S 2 we recall the required notation and  
facts about supersolvability of arrangements
and reflection arrangements from
\cite[\S 4, \S6]{orlikterao:arrangements}.
Further, in Proposition \ref{prop:product-super},
we show that supersolvable arrangements
behave well with respect to 
the  product construction for arrangements.
Using  this fact,
it is easy to construct non-supersolvable arrangements
admitting modular elements of every 
possible rank, Remark \ref{rem:reducible}.

Theorems \ref{super} and \ref{super-rank2} are proved in \S \ref{sect:proof}
in a sequence of lemmas.
Here we provide a particularly useful inductive tool for 
showing that the  
reflection arrangement 
of a given  
reflection group $W$  is not supersolvable
provided $W$ admits a suitable parabolic subgroup 
whose reflection arrangement is not supersolvable, 
cf.\ Lemma \ref{lem:parabolicsuper}.

For general information about arrangements and reflection groups we refer
the reader to \cite{bourbaki:groupes} and \cite{orlikterao:arrangements}.

\section{Recollections and Preliminaries}

\subsection{Hyperplane Arrangements}
\label{ssect:hyper}

Let $V = \BBC^\ell$ 
be an $\ell$-dimensional complex vector space.
A \emph{hyperplane arrangement} is a pair
$(\CA, V)$, where $\CA$ is a finite collection of hyperplanes in $V$.
Usually, we simply write $\CA$ in place of $(\CA, V)$.
We only consider central arrangements.
The empty arrangement in $V$ is denoted by $\Phi_\ell$.

The \emph{lattice} $L(\CA)$ of $\CA$ is the set of subspaces of $V$ of
the form $H_1\cap \dotsm \cap H_n$ where $\{ H_1, \ldots, H_n \}$ is a subset
of $\CA$. Note that $V$ belongs to $L(\CA)$
as the intersection of the empty 
collection of hyperplanes.
The lattice $L(\CA)$ is a partially ordered set by reverse inclusion:
$X \le Y$ provided $Y \subseteq X$ for $X,Y \in L(\CA)$.
We have a \emph{rank} function on $L(\CA)$: $r(X) := \codim_V(X)$.
The \emph{rank} $r(\CA)$ of $\CA$ is the rank of a maximal element in $L(\CA)$ with respect
to the partial order.
With this definition $L(\CA)$ is a geometric lattice, 
\cite[p.\ 24]{orlikterao:arrangements}.
The $\ell$-arrangement $\CA$ is called \emph{essential} provided $r(\CA) = \ell$.

The \emph{product}
$\CA = (\CA_1 \times \CA_2, V_1 \oplus V_2)$ 
of two arrangements $(\CA_1, V_1), (\CA_2, V_2)$
is defined by
\begin{equation*}
\label{eq:product}
\CA := \CA_1 \times \CA_2 = \{H_1 \oplus V_2 \mid H_1 \in \CA_1\} \cup 
\{V_1 \oplus H_2 \mid H_2 \in \CA_2\},
\end{equation*}
see \cite[Def.\ 2.13]{orlikterao:arrangements}.
Let $\CA = \CA_1 \times \CA_2$ be the product
of the two arrangements $\CA_1$ and $\CA_2$. 
We define  a partial order on 
$L(\CA_1) \times L(\CA_2)$ by
$(X_1,X_2) \le (Y_1,Y_2)$ provided $X_1 \le Y_1$ and $X_2 \le Y_2$.
Then, by  \cite[Prop.\ 2.14]{orlikterao:arrangements},
there is a lattice isomorphism
\begin{equation}
\label{eq:product-lattice}
 L(\CA_1) \times L(\CA_2) \cong L(\CA) \quad \text{by} \quad
(X_1, X_2) \mapsto X_1 \oplus X_2.
\end{equation}


Note that 
$\CA \times \Phi_0 = \CA$
for any arrangement $\CA$. 
If $\CA$ is of the form $\CA = \CA_1 \times \CA_2$, where 
$\CA_i \ne \Phi_0$ for $i=1,2$, then $\CA$
is called \emph{reducible}, else $\CA$
is 
\emph{irreducible}, 
\cite[Def.\ 2.15]{orlikterao:arrangements}.
For instance, the braid arrangement $\CA(A_\ell)$ is 
the product of the empty $1$-arrangement and an irreducible $(\ell-1)$-arrangement, 
\cite[Ex.\ 2.16]{orlikterao:arrangements}.

\subsection{Reflection Groups and Reflection Arrangements}
\label{ssect:refl}
The irreducible finite complex reflection groups were 
classified by Shephard and Todd, \cite{sheppardtodd}.
Let $W  \subseteq \GL(V)$ be a finite complex reflection group.
For $w \in W$, we write 
$\Fix(w) :=\{ v\in V \mid w v = v\}$ for 
the fixed point subspace of $w$.
For $U \subseteq V$ a subspace, we 
define the \emph{parabolic subgroup}
$W_U$ of $W$ by 
$W_U := \{w \in W \mid U \subseteq \Fix(w)\}$.

The \emph{reflection arrangement} $\CA = \CA(W)$ of $W$ in $V$ is 
the hyperplane arrangement 
consisting of the reflecting hyperplanes of the elements in $W$
acting as reflections on $V$.
By Steinberg's Theorem \cite[Thm.\ 1.5]{steinberg:invariants},
for $U \subseteq V$ a subspace, 
the parabolic subgroup
$W_U$ is itself a complex reflection group,
generated by the unitary reflections in $W$ that are contained
in $W_U$. 
This allows us to identify the 
reflection arrangement $\CA(W_U)$
of $W_U$ as a subarrangement of $\CA$.
This way, the lattice 
$L(\CA(W_U))$ of $\CA(W_U)$ is identified with a sublattice of $L(\CA)$.
We make these identifications throughout.

Note that for $X \in L(\CA)$, we have 
$\CA(W_X) = \CA_X := \{H \in \CA \mid X \subseteq H\}$,
cf.\ 
\cite[Thm.\ 6.27, Cor.\ 6.28]{orlikterao:arrangements}.
It follows that 
$L(\CA(W_X)) = L(\CA_X) = L(\CA)_X  
:= \{Z \in L(\CA) \mid X \subseteq Z\}$,
by \cite[Lem.\ 2.11]{orlikterao:arrangements}.

Following \cite[\S 6.4, App.\ C]{orlikterao:arrangements},
we use the convention 
to label the $W$-orbit 
of $X \in L(\CA)$ by the type
of the complex reflection group $W_X$.
In that way we get a correspondence between
the $W$-orbits in $L(\CA)$ and the conjugacy classes of 
parabolic subgroups of $W$.

\subsection{Free and Inductively Free Arrangements}
\label{ssect:free}

Let $S = S(V^*)$ be the symmetric algebra of the dual space $V^*$ of $V$.
If $x_1, \ldots , x_\ell$ is a basis of $V^*$, then we identify $S$ with 
the polynomial ring $\BBC[x_1, \ldots , x_\ell]$.
Letting $S_p$ denote the $\BBC$-subspace of $S$
consisting of the homogeneous polynomials of degree $p$ (along with $0$),
we see that
$S$ is naturally $\BBZ$-graded: $S = \oplus_{p \in \BBZ}S_p$, where
$S_p = 0$ for $p < 0$.

Let $\Der(S)$ be the $S$-module of $\BBC$-derivations of $S$.
For $i = 1, \ldots, \ell$, 
let $D_i := \partial/\partial x_i$.
Then $D_1, \ldots, D_\ell$ is a $\BBC$-basis of $\Der(S)$.
We say that $\theta \in \Der(S)$ is 
\emph{homogeneous of polynomial degree p}
provided 
$\theta = \sum_{i=1}^\ell f_i D_i$, 
where $f_i \in S_p$ for each $1 \le i \le \ell$.
In this case we write $\pdeg \theta = p$.
Let $\Der(S)_p$ be the $\BBC$-subspace of $\Der(S)$ consisting 
of all homogeneous derivations of polynomial degree $p$.
Then $\Der(S)$ is a graded $S$-module:
$\Der(S) = \oplus_{p\in \BBZ} \Der(S)_p$.

Following \cite[Def.~4.4]{orlikterao:arrangements}, 
for $f \in S$, we define the $S$-submodule $D(f)$ of $\Der(S)$ by
$D(f) := \{\theta \in \Der(S) \mid \theta(f) \in f S\} $.
Let $\CA$ be an arrangement in $V$. 
Then for $H \in \CA$ we fix $\alpha_H \in V^*$ with
$H = \ker \alpha_H$.
The \emph{defining polynomial} $Q(\CA)$ of $\CA$ is given by 
$Q(\CA) := \prod_{H \in \CA} \alpha_H \in S$.

The \emph{module of $\CA$-derivations} of $\CA$ is 
defined by 
$D(\CA) := D(Q(\CA))$.
We say that $\CA$ is \emph{free} if the module of $\CA$-derivations
$D(\CA)$ is a free $S$-module.
The notion of freeness was introduced by Saito in his 
seminal work \cite{saito}.

With the $\BBZ$-grading of $\Der(S)$, the module of $\CA$-derivations
becomes a graded $S$-module $D(\CA) = \oplus_{p\in \BBZ} D(\CA)_p$,
where $D(\CA)_p = D(\CA) \cap \Der(S)_p$, 
\cite[Prop.\ 4.10]{orlikterao:arrangements}.
If $\CA$ is a free arrangement, then the $S$-module 
$D(\CA)$ admits a basis of $\ell$ homogeneous derivations, 
say $\theta_1, \ldots, \theta_\ell$, \cite[Prop.\ 4.18]{orlikterao:arrangements}.
While the $\theta_i$'s are not unique, their polynomial 
degrees $\pdeg \theta_i$ 
are unique (up to ordering). This multiset is the set of 
\emph{exponents} of the free arrangement $\CA$
and is denoted by $\exp \CA$.

There is a stronger notion of freeness, motivated
by the so called 
\emph{Addition-Deletion Theorem},
see \cite[Thm.\ 4.51]{orlikterao:arrangements}.

\begin{defn}
\label{def:indfree}
The class $\CIF$ of \emph{inductively free} arrangements 
is the smallest class of arrangements subject to
\begin{itemize}
\item[(i)] $\Phi_\ell \in \CIF$ for each $\ell \ge 0$;
\item[(ii)] if there exists an $H \in \CA$ such that both
the subarrangement 
$\CA \setminus\{H\}$ of $\CA$ and 
the restriction of $\CA$ to $H$, 
$\CA^H := \{H' \cap H \mid H' \in \CA \setminus\{H\}\}$, 
belong to $\CIF$, and $\exp \CA^H \subseteq \exp (\CA \setminus\{H\})$, 
then $\CA$ also belongs to $\CIF$.
\end{itemize}
\end{defn}

Terao \cite{terao:freeI} proved that 
each reflection arrangement is free.
In \cite{hogeroehrle:indfree}, we classified all 
inductively free reflection arrangements;
cf.\ Theorem \ref{thm:ind-free}.

\subsection{Supersolvable Arrangements}
\label{ssect:supersolve}

Let $\CA$ be an arrangement.
Following \cite[\S 2]{orlikterao:arrangements}, we say
that $X \in L(\CA)$ is \emph{modular}
provided $X + Y \in L(\CA)$ for every $Y \in L(\CA)$.
(This is not the actual definition of a modular
element but it is equivalent to the definition in our case, 
 \cite[Cor.\ 2.26]{orlikterao:arrangements}.)
Let $\CA$ be a central (and essential) $\ell$-arrangement.
The following notion is due to Stanley \cite{stanley:super}. 

\begin{defn}
\label{def:super}
We say that $\CA$ is 
\emph{supersolvable} 
provided there is a maximal chain
\[
V = X_0 < X_1 < \ldots < X_{\ell-1} < X_\ell = \{0\}
\]
 of modular elements $X_i$ in $L(\CA)$,
cf.\ \cite[Def.\ 2.32]{orlikterao:arrangements}.
\end{defn}

This terminology owes to the fact that 
the lattice of subgroups of a finite supersolvable
group satisfies the condition in Definition \ref{def:super}.

\begin{remark}
\label{rem:2-arr}
By \cite[Ex.\ 2.28]{orlikterao:arrangements}, 
$V$, $\{0\}$ and the members in $\CA$ 
are always modular in $L(\CA)$.
It follows  
that all $0$- $1$-, and $2$-arrangements are supersolvable.
\end{remark}

As mentioned in the Introduction, 
supersolvable arrangements
are always inductively free, see also 
\cite[Thm.\ 4.58]{orlikterao:arrangements}.
In general, a free $3$-arrangement need not be 
supersolvable as such an arrangement need not be 
inductively free,
see \cite[Ex.\ 4.59]{orlikterao:arrangements}.

The following observation is immediate from 
Definition \ref{def:super} and
Remark \ref{rem:2-arr}.

\begin{lemma}
\label{lem:super-rank2}
A $3$-arrangement $\CA$ is supersolvable if and only if 
there exists a modular rank $2$ element in $L(\CA)$.
\end{lemma}

Thanks to \cite[Prop.\ 4.28]{orlikterao:arrangements}, 
free arrangements behave well with respect to 
the  product construction for arrangements.
This is also the case for  
supersolvable arrangements. 

\begin{proposition}
\label{prop:product-super}
Let $\CA_1, \CA_2$ be two arrangements.
Then  $\CA = \CA_1 \times \CA_2$ is supersolvable
if and only if both 
$\CA_1$ and $\CA_2$ are supersolvable and in that case
the multiset of exponents of $\CA$ is given by 
$\exp \CA = \{\exp \CA_1, \exp \CA_2\}$.
\end{proposition}

\begin{proof}
Let $\CA_i = (\CA_i, V_i)$, $\ell_i = \dim V_i$
for $i = 1,2$, 
and let $(\CA,V) = (\CA_1 \times \CA_2, V_1 \oplus V_2)$,
and $\ell = \ell_1 + \ell_2 = \dim V$. 

First suppose that both $\CA_1$ and $\CA_2$
are supersolvable with
\[
V_i = X_i^0 < X_i^1 < \ldots < X_i^{\ell_i} = \{0\}
\]
being a maximal chain of modular elements in $L(\CA_i)$ for $i = 1,2$. 
For $0 \le j \le \ell_1$ set
$Z^j := X_1^j \oplus V_2$
and for $\ell_1 + 1 \le j \le \ell$ set
$Z^j := \{0\} \oplus X_2^{j-\ell_1} = X_2^{j-\ell_1}$.
By \eqref{eq:product-lattice}, $Z^j \in L(\CA)$ for each $0 \le j \le \ell$. 
We claim that 
\[
V = Z^0 < Z^1 < \ldots < Z^{\ell} = \{0\}
\]
is a maximal chain of modular elements in $L(\CA)$.
Clearly, this is a proper chain of elements in $L(\CA)$
of length $\ell$ by construction, cf.\ \eqref{eq:product-lattice}. 
Let $Y = Y_1 \oplus Y_2 \in L(\CA)$.
Then, for $0 \le j \le \ell_1$, 
\[
Z^j  + Y = (X_1^j \oplus V_2) + (Y_1 \oplus Y_2) = (X_1^j + Y_1)\oplus V_2
\] 
belongs to $L(\CA)$, by \eqref{eq:product-lattice},
since $X_1^j$ is modular in $L(\CA_1)$, so that $X_1^j + Y_1 \in L(\CA_1)$.
Likewise,  
\[
Z^j  + Y = (\{0\} \oplus X_2^{j-\ell_1}) + (Y_1 \oplus Y_2) = 
Y_1 \oplus (X_2^{j-\ell_1} + Y_2) 
\] 
lies in $L(\CA)$, 
for $\ell_1 + 1 \le j \le \ell$, by \eqref{eq:product-lattice}.
As $Y \in L(\CA)$ is arbitrary, 
$\CA$ is supersolvable.

Now suppose that $\CA$ is supersolvable. Let 
\[
V = X^0 < X^1 < \ldots < X^{\ell-1} < X^\ell = \{0\}
\]
be a maximal chain of modular elements in $L(\CA)$.
Then $X^j = X_1^j \oplus X_2^j$ for each $1 \le j \le \ell$, by \eqref{eq:product-lattice}.
As each $X^j$ is modular in $L(\CA)$, for every $Y = Y_1 \oplus Y_2 \in L(\CA)$, the sum
\[
X^j + Y = (X_1^j \oplus X_2^j) + (Y_1 \oplus Y_2) = (X_1^j + Y_1 ) \oplus (X_2^j + Y_2)
\]
belongs to $L(\CA)$. By \eqref{eq:product-lattice}, it follows that 
 $X_i^j + Y_i$ belongs to $L(\CA_i)$ for every  $1 \le j \le \ell$ and $i = 1,2$.
Since $Y \in L(\CA)$ is arbitrary, 
each $X_i^j$ is modular in $L(\CA_i)$ for every  $1 \le j \le \ell$ and $i = 1,2$.
Since $\ell = \ell_1 + \ell_2$, it follows from our construction that there are subsequences 
$1 \le j_1 < \ldots < j_{\ell_1}$ and 
$1 \le k_1 < \ldots < k_{\ell_2}$
of the integers from $1$ to $\ell$ such that 
\[
V_1 = X_1^0 < X_1^{j_1} < \ldots  
< X_1^{j_{\ell_1}} = \{0\}
\]
and
\[
V_2 = X_2^0 < X_2^{k_1} < \ldots 
< X_2^{k_{\ell_2}} = \{0\}
\]
are maximal chains of modular elements in $L(\CA_1)$ and  $L(\CA_2)$,
respectively. Thus both $\CA_1$ and $\CA_2$ are supersolvable,
as desired.

Finally, the statement on the exponents 
follows from \cite[Prop.\ 4.28]{orlikterao:arrangements}
and the fact that supersolvable arrangements are free,
\cite[Thm.\ 4.2]{jambuterao:free}.
\end{proof}

Our final observation in this section shows that 
it is easy to construct an arrangement $\CA$ which 
is  not supersolvable but still 
admits modular elements of every possible rank.

\begin{remark}
\label{rem:reducible}
Suppose that 
$\CA_i = (\CA_i, V_i)$ (for $i = 1,2$) are arrangements so that 
$\CA_1$ is supersolvable but $\CA_2$ is not and that 
$\ell_1  = \dim V_1 \ge \ell_2 = \dim V_2$.
Consider the product $(\CA,V) = (\CA_1 \times \CA_2, V_1 \oplus V_2)$.
While $\CA$ is again not supersolvable,
thanks to  Proposition \ref{prop:product-super},
it is easy to see that $L(\CA)$ admits modular elements of
every possible rank $r$
for any $0 \le r \le \ell_1 + \ell_2 = r(\CA)$.
For, let 
\[
V_1 = X_0 < X_1 < \ldots < X_{\ell_1} = \{0\}
\]
be a maximal chain of modular elements in $L(\CA_1)$.
Then, by \eqref{eq:product-lattice},
$Z_s := X_s \oplus V_2$  and
$Z'_t := X_t \oplus \{0\}$ 
belong to $L(\CA)$ for each $0 \le s \le  \ell_2 \le \ell_1$
and $1 \le t \le \ell_1$.
Since  $\codim_V Z_s = \codim_{V_1} X_s = s$ and 
$\codim_V Z'_t = \codim_{V_1} X_t +\dim V_2 = t + \ell_2$,
the rank of $Z_s$ is $s$ 
for each $0 \le s \le \ell_2$
and that of 
$Z'_t$ is $\ell_2 +t$
for each $1 \le t \le \ell_1$. 
Now let $Y = Y_1 \oplus Y_2 \in L(\CA)$ be arbitrary.
Then, since each $X_s$ is modular  in $L(\CA_1)$, 
it follows from \eqref{eq:product-lattice}
that
\[
Z_s + Y = ( X_s \oplus V_2) + ( Y_1 \oplus Y_2) = (X_s +Y_1) \oplus V_2 \in L(\CA)
\]
and 
\[
Z'_t + Y = ( X_t \oplus \{0\}) + (Y_1 \oplus Y_2) = (X_t +Y_1) \oplus Y_2 \in L(\CA).
\]
Whence $Z_s$ is modular of rank $s$ 
for $0 \le s \le \ell_2$ and 
$Z'_t$ is modular of rank $\ell_2 + t$
for $1 \le t \le \ell_1$. 
In particular, 
$L(\CA)$ admits modular elements of 
every possible rank.
\end{remark}

\section{Proofs of Theorems \ref{super} and \ref{super-rank2}}
\label{sect:proof}

Our first result allows us to conclude that $\CA(W)$ 
is not supersolvable given that the reflection arrangement of a 
suitable parabolic subgroup of $W$ is not supersolvable.
While this is an elementary observation, it is 
nevertheless a rather effective inductive tool.

\begin{lemma}
\label{lem:parabolicsuper}
Let $W$ be a complex reflection group and $\CA = \CA(W)$ its reflection
arrangement. 
Suppose that there are $X \in L(\CA)$ and $r \in \BBN$ such that
$L(\CA(W_X))$ does not admit modular elements of rank $r$ 
and that 
every $W$-orbit of elements in $L(\CA)$ of rank $r$ meets $L(\CA(W_X))$.
Then $L(\CA)$ does not admit modular elements of rank $r$.
In particular, $\CA$ is not supersolvable.
\end{lemma}

\begin{proof}
Let $Y \in L(\CA)$ be of rank $r$.
Recall from Section \ref{ssect:refl} that 
$L(\CA(W_X)) = L(\CA)_X$.
Therefore, by our hypotheses there exists a $w \in W$, so that
$wY \in L(\CA)_X$, and there is a $Z \in L(\CA)_X$ so that 
$U := wY + Z \notin L(\CA)_X$.
It suffices to show that 
$U \notin L(\CA)$, as then 
$Y + w\inverse Z$ does not belong to $L(\CA)$ and so $Y$ is not modular.
So, for a contradiction, suppose that $U \in L(\CA)$.
Since both $wY$ and $Z$ belong to $L(\CA)_X$,
we have $X \subseteq wY + Z = U$ and so 
$U \in L(\CA)_X$
which is absurd, since $U \notin L(\CA)_X$, by construction.
Consequently, $U \notin L(\CA)$, as desired.
\end{proof}

Note, the first condition on $L(\CA(W_X))$
in Lemma \ref{lem:parabolicsuper} says
that $\CA(W_X)$ is not supersolvable.

\begin{lemma}
\label{lem:super}
If $W$ is of type $D_4$, $F_4$, $H_3$, or $W = G_{25}$ or $G_{26}$,
then there are no modular elements in $L(\CA(W))$ of rank $2$.
In particular, $\CA(W)$ is not supersolvable.
\end{lemma}

\begin{proof}
Let $\CA = \CA(W)$.
Using the explicit data on the $W$-orbits in $L(\CA)$
from \cite[\S 6, App.\ C]{orlikterao:arrangements}, 
we choose an orbit representative $X$
for each $W$-orbit  in $L(\CA)$ of elements of rank $2$, 
and give an explicit  $Y \in L(\CA)$ in each case
so that $X + Y \notin L(\CA)$.
We denote the coordinate functions in $S$ simply by $a, b, c$, etc.\ 
and for $f \in V^*$, we write $H_f$ for $\ker f$.

(i). Let $W$ be of type $D_4$.
Jambu and Terao have already observed
that $\CA = \CA(D_4)$ is not supersolvable,  \cite[Ex.\ 5.5]{jambuterao:free}.
We give a different, elementary argument showing that 
$L(\CA)$ does not admit modular elements of rank $2$.
The defining polynomial of $\CA = \CA(D_{4})$ is
\begin{align*}
Q(D_4) &:= (a - b)(a + b)(a - c)(a + c)(a - d)(a + d)\\
& (b - c)(b + c)(b - d)(b + d)(c - d)(c + d).
\end{align*}
There are four $W$-orbits of elements of rank $2$ in $L(\CA)$, 
corresponding to three conjugacy classes of 
parabolic subgroups of $W$ of type $A_1^2$ and one of type $A_2$.
The conjugacy classes of  type $A_1^2$ are fused by the action of the
group of graph automorphisms of $W$.
It thus suffices to only consider one orbit in $L(\CA)$ with 
parabolic subgroup of type $A_1^2$ and another one 
with point stabilizer of type $A_2$,
where representatives of these orbits are given by 
\[
X_1  = H_{a+b} \cap H_{a-b} \quad \text{ and } \quad
X_2  = H_{a-b} \cap H_{b-c},
\]
respectively.
One readily checks that 
$X_1 + (H_{b+d} \cap H_{b-d}) = H_{b}\notin \CA$, 
as well as that 
$X_2 + (H_{a+b} \cap H_{c-d}\cap H_{c+d})  
= H_{a+b-2c} \notin \CA$.
Thus both $X_1$ and $X_2$ are not modular.

(ii). Let $W$ be of type $F_4$. 
The defining polynomial of $\CA = \CA(F_4)$ is given by 
\begin{align*}
Q(F_4) & := 
abcd(a + b)(b + c)(c + d)(b + 2c)(a + b + c)(b + c + d)(a + b + 2c) \\
& (a + b + c + d)(b + 2c + d)(a + 2b + 2c)(a + b + 2c + d)(b + 2c + 2d) \\
& (a + 2b + 2c + d)(a + b + 2c + 2d)(a + 2b + 3c + d)(a + 2b + 2c + 2d) \\
& (a + 2b + 3c + 2d)(a + 2b + 4c + 2d)(a + 3b + 4c + 2d)(2a + 3b + 4c + 2d).
\end{align*}
There are four $W$-orbits of elements of rank $2$ 
in $L(\CA)$, cf.\ \cite[Table C.9]{orlikterao:arrangements},
corresponding to the four conjugacy classes of 
parabolic subgroups of $W$ of type $A_2$, $\widetilde A_2$, 
$A_1 \times \widetilde A_1$, and $B_2$ with 
representatives 
\[
X_1  = H_a \cap H_b, \quad 
X_2  = H_c \cap H_d, \quad 
X_3  = H_{c+d} \cap H_{a + 2b + 2c + 2d}, \text{ and }\quad
X_4  = H_a \cap H_{b + c}, 
\]
respectively.
We claim that none of these elements is modular.
First note that we have $X_1 + X_3 = H_{a+2b} \notin \CA$, 
so both $X_1$ and $X_3$ are not modular.
Moreover, one checks that   
$X_2 + (H_{a + 2b + 3c + d} \cap H_{a + 2b + 2c + 2d}) 
= H_{c-d} \notin \CA$, and that 
$X_4 + (H_{b} \cap H_{a + b + c + d} \cap H_{a + 2b + 4c + 2d}) 
= H_{a-2b-2c} \notin \CA$, and so $X_2$ and $X_4$ are not modular either.

(iii). Let $W$ be of type $H_3$.
The defining polynomial of $\CA = \CA(H_3)$ is
\begin{align*}
Q(H_3) &  :=  abc(a - \omega b)(a -(\omega+1)b)
(b + c)(a + b)(a - \omega b - \omega c) \\
& (a - (\omega+1)b -(\omega+1)c)
(a + b + c)
(a -\omega b -(\omega+1)c)
(a -\omega b + c) \\
& (a + b + (\omega+2)c)
(a + b - (\omega+1)c)
(a -2(\omega+1)b -(\omega+1)c),
\end{align*} 
where $\omega = \eta^2 + \eta^3$ and $\eta$ is a primitive 
$5$-th root of unity.
There are three $W$-orbits of elements of rank $2$ in $L(\CA)$, cf.\ 
\cite[Table C.4]{orlikterao:arrangements}. 
They correspond to the three conjugacy classes of
parabolic subgroups of $W$ of types 
$I_2(5)$, $A_2$ and $A_1^2$ with
respective representatives
\[
X_1  = H_{a} \cap H_{b}, \quad  
X_2  = H_{a}\cap H_{a -\omega b -(\omega+1)c}, \text{ and } \quad 
X_3  = H_{c} \cap H_{a -2(\omega+1)b -(\omega+1)c}.
\]
One calculates that $X_1 + X_3 = H_{a-2(\omega+1)b} \notin \CA$, 
and moreover that
$X_2 + (H_{a + b} \cap H_{a -\omega b + c}) 
= H_{2a- (\omega-1)b + c} \notin \CA$, 
whence none of these elements is modular.

(iv). Let $W = G_{25}$.
The defining polynomial of $\CA = \CA(G_{25})$ is
\begin{align*}
Q(G_{25}) 
& := abc(a + b + c)(a + b + \zeta c)(a + b +\zeta^2 c)(a + \zeta b + c) \\
& (a + \zeta b + \zeta c)(a + \zeta b + \zeta^2 c) (a + \zeta^2 b + c)(a + \zeta^2 b + \zeta c)(a + \zeta^2 b + \zeta^2 c),
\end{align*} 
where $\zeta$ is a primitive $3$-rd root of unity.
There are two $W$-orbits of elements of rank $2$ in $L(\CA)$, 
\cite[Table C.6]{orlikterao:arrangements}. 
They correspond to the two conjugacy classes of
parabolic subgroups of $W$ isomorphic to 
$C(3)^2$ and $G_4$ with respective
representatives 
\[
X_1 = H_{a} \cap H_{b}  \text{ and } \quad
X_2  =  H_{c} \cap H_{a+b+c}. 
\]
Since $X_1 + X_2 = H_{a+b}\notin \CA$,
both $X_1$ and $X_2$ are not modular.

(v).
Let $W = G_{26}$. 
Falk and Randell \cite[(3.3)]{falkRandell:homotopy}
already asserted that $\CA = \CA(G_{26})$ is not supersolvable.
We give a different, elementary argument showing that 
$L(\CA)$ does not admit modular elements of rank $2$.
The defining polynomial of $\CA$ is
\begin{align*}
Q(G_{26}) &:= 
abc (a - b)(a - c)(b - c)
(a -\zeta b)(a - \zeta^2 b)
(a -\zeta c)(a -\zeta^2 c)
(b -\zeta c)(b - \zeta^2 c)\\
&(a + b + c) 
(a + b + \zeta c)
(a + b + \zeta^2 c)
(a + \zeta b + c)
(a + \zeta b + \zeta c)
(a + \zeta b + \zeta^2 c)\\
&(a + \zeta^2 b + c)
(a + \zeta^2 b + \zeta c)
(a + \zeta^2 b + \zeta^2 c),
\end{align*}
where $\zeta$ is a primitive $3$-rd root of unity.
There are three $W$-orbits of elements of rank $2$ in $L(\CA)$, 
\cite[Table C.7]{orlikterao:arrangements}. 
They correspond to the three conjugacy classes of
parabolic subgroups of $W$ isomorphic to 
$A_1 \times C(3)$, $G_4$, and $G(3,1,2)$
with representatives 
\[
X_1  = H_{b} \cap H_{a-\zeta c}, \quad
X_2  = H_{c} \cap H_{a + \zeta b + c},  \text{ and } \quad 
X_3 = H_a \cap H_b, 
\]
respectively.
One checks that 
$X_2 + X_3 = H_{a+\zeta b}\notin \CA$,
and that 
$X_1 + (H_{a-b} \cap H_{b-\zeta^2 c}) 
= H_{a-(\zeta+2)b - \zeta c} \notin \CA$, 
and so none of these elements is modular.

This completes the proof of Lemma \ref{lem:super}.
\end{proof}

Armed with Lemmas \ref{lem:parabolicsuper}
and \ref{lem:super},
we can now readily derive that a large class of 
unitary reflection groups does not admit 
a supersolvable reflection arrangement.
For $W$ of type 
$E_6$, $E_7$,  $E_8$, and $H_4$,
Falk and Randell  had raised 
this question in \cite[(2.1)]{falkRandell:homotopy}.

\begin{lemma}
\label{lem:inductionsuper}
If $W$ is of type $D_\ell$, for $\ell \ge 5$,
$E_6$, $E_7$,  $E_8$, $H_4$, or $W= G_{32}$,
then there are no modular elements in $L(\CA(W))$ of rank $2$.
In particular, $\CA(W)$ is not supersolvable.
\end{lemma}

\begin{proof}
Let $\CA = \CA(W)$.
If $W$ is one of the Weyl groups as in the statement, let $W_X$ be a
parabolic subgroup of $W$ of type $D_4$, if $W = H_4$, let $W_X = H_3$,
and if $W= G_{32}$, then let $W_X = G_{25}$. 
By Lemma \ref{lem:super}, 
$L(\CA(W_X))$ does not admit modular elements of rank $2$.
{}From the information on the $W$-orbits in $L(\CA)$ given in 
\cite[\S 6, App.\ C]{orlikterao:arrangements},
we infer that with these choices, 
every $W$-orbit of elements in $L(\CA)$ of rank $2$ 
meets the sublattice $L(\CA(W_X))$.
The desired result follows from 
Lemma \ref{lem:parabolicsuper}.
\end{proof}

\begin{proof}[Proof of Theorem \ref{super}]
Thanks to Proposition \ref{prop:product-super},
the question of supersolvability reduces 
to the case when $\CA = \CA(W)$ is irreducible.
Therefore, we may assume that $W$ is irreducible.

(i).
Let $W$ be of rank at most $2$.
Then $\CA(W)$ is supersolvable,
by Remark \ref{rem:2-arr}. 

(ii).
Let $W$ be an irreducible Coxeter group of rank at least $3$.
Fadell and Neuwirth \cite{fadellneuwirth}
proved that the
braid arrangement $\CA(A_\ell)$ is fiber-type 
and Brieskorn \cite{brieskorn:tresses} 
showed that the reflection arrangement of the
hyperoctahedral group $\CA(B_\ell)$ is also fiber-type.
Thus $\CA(A_\ell)$  and $\CA(B_\ell)$ 
are supersolvable, by \cite{terao:modular}.
It follows from Lemmas \ref{lem:super}
and \ref{lem:inductionsuper}
that for all other types 
$\CA(W)$ is not supersolvable.

(iii). Let $W$ be a monomial group $G(r,p,\ell)$.
Note that the reflection arrangements for 
$G(r,1,\ell)$ and  $G(r,p,\ell)$ for $p \ne r$ coincide.
It thus suffices to consider the case $G(r,1,\ell)$.
Let $\CA = \CA(G(r,1,\ell))$
for $r, \ell \ge 3$.
Orlik and Solomon \cite[\S 4]{orliksolomon:coxeter} showed
that $\CA$ is fiber-type,
and so $\CA$ is supersolvable.
We give an elementary, direct argument for the supersolvability of $\CA$.
Let $H_i = \ker x_i$ 
and define $X_i = H_1 \cap H_2 \cap \ldots \cap H_i$
for $1 \le i \le \ell$. Set $X_0 = V$.
We claim that
\[
V = X_0 < X_1 < \ldots < X_\ell = \{0\}
\]
is a maximal chain of modular elements in $L(\CA)$.
Let $\zeta = e^{2\pi i/r}$ be a primitive $r$-th root of unity.
The hyperplanes of $\CA$ are of the form
$H_i =\ker x_i$ for $1 \le i \le \ell$
and 
$H_{i,j}(m) = \ker(x_i - \zeta^m x_j)$, 
for $1\le i < j  \le \ell$ and $0 \le m \le r -1$.
So for  an intersection of such 
hyperplanes $Y \in L(\CA)$,  the coordinates of 
the subspace $Y$ are either $0$, some coincide,
while others differ by a power of $\zeta$.
Now $X_k$ consists of all vectors in $V$ whose first $k$ coordinates
are zero and all others are arbitrary, for $1\le k \le \ell$.
Thus the relations among the coordinates of a vector in the 
sum $X_k + Y$ of the two subspaces
are the same as for the ones in $Y$ for the first $k$
coordinates and for higher indices there are no restrictions.
So $X_k + Y$ coincides with the intersection
of those hyperplanes $H_i$ and $H_{i,j}(m)$ that define $Y$ 
subject to the requirement  that all occurring indices $i,j$ are at most $k$.
In particular, $X_k + Y$ belongs again to $L(\CA)$
and so $X_k$ is modular  for $1\le k \le \ell$, as claimed.

Note that this argument also applies to the 
braid arrangement $\CA(G(1,1,\ell))$ 
(cf.\ \cite[Ex.\ 2.33]{orlikterao:arrangements}), 
as well as to the reflection arrangement of the
hyperoctahedral group $\CA(G(2,1,\ell))$.

By Theorem \ref{thm:ind-free}  and \cite[Thm.\ 4.2]{jambuterao:free}, 
the reflection arrangements 
of the monomial groups $G(r,r,\ell)$ for $r,\ell \ge 3$ 
are not supersolvable.

(iv).
Finally, let $W$ be an irreducible, non-real, unitary 
reflection group of rank at least $3$.
By Theorem \ref{thm:ind-free} and \cite{jambuterao:free}, 
the reflection arrangements 
of $G_{24}, G_{27}, G_{29}, G_{31}, G_{33}$, and  $G_{34}$
are not supersolvable, neither are the ones of
$G_{25}$, $G_{26}$ and $G_{32}$,
by Lemmas \ref{lem:super} and \ref{lem:inductionsuper}.

This completes the proof of Theorem \ref{super}.
\end{proof}

We now attend to 
Theorem \ref{super-rank2}.
Our first observation is immediate from 
 Theorem \ref{super} and Lemma \ref{lem:super-rank2}.

\begin{lemma}
\label{lem:super2}
If $W$ is a monomial group $G(r,r,3)$ for $r \ge 3$, $G_{24}$ or $G_{27}$,
then there are no modular elements in $L(\CA(W))$ of rank $2$.
\end{lemma}

\begin{lemma}
\label{lem:super3}
If $W$ is a monomial group $G(r,r,4)$ for $r \ge 3$, $G_{29}$ or $G_{31}$,
then there are no modular elements in $L(\CA(W))$ of rank $2$.
\end{lemma}

\begin{proof}
Let $\CA = \CA(W)$.
As before, we denote the coordinate functions in $S$ simply by $a, b, c$, and $d$
and for $f \in V^*$, we write $H_f$ for $\ker f$.

(i). Let $W = G(r,r,4)$ for $r \ge 3$.
There are three $W$-orbits in $L(\CA)$ of elements of rank $2$
corresponding to the conjugacy classes of 
parabolic subgroups of types $A_1^2$, $A_2$ and $I_2(r)$,
cf.\ \cite[\S 6.4]{orlikterao:arrangements}.
Let $W_X = G(r,r,3)$ 
be a parabolic subgroup, where $X \in L(\CA)$ is of rank $3$.
There are two $W_X$-orbits in $L(\CA(W_X))$ of elements of rank $2$
corresponding to the conjugacy classes of 
parabolic subgroups of types $A_2$ and $I_2(r)$.
It follows from  Lemma \ref{lem:super2} and the argument in the proof of
Lemma \ref{lem:parabolicsuper} that for each $Y \in L(\CA)$ of rank $2$ whose
$W$-orbit meets $L(\CA(W_X))$, $Y$ is not modular.
This leaves the $W$-orbit in $L(\CA)$
whose elements have point stabilizer of type $A_1^2$ 
to be considered. 
One such element is given by $Y = H_{a-b} \cap H_{c-d}$.
But $Y + (H_{b-c} \cap H_{a-d}) = H_{a-b+c-d} \notin \CA$, so $Y$ is
not modular.

(ii). Let $W = G_{29}$.
There are three $W$-orbits in $L(\CA)$ of elements of rank $2$
corresponding to the conjugacy classes of 
parabolic subgroups of types $A_1^2$, $A_2$ and $B_2$,
cf.\ \cite[Table C.10]{orlikterao:arrangements}.
Let $W_X = G(4,4,3)$ 
be a parabolic subgroup for, where $X$ has rank $3$.
There are two $W_X$-orbits in $L(\CA(W_X))$ of elements of rank $2$
corresponding to the conjugacy classes of 
parabolic subgroups of types $A_2$ and $B_2$.
It follows from  Lemma \ref{lem:super2} and the argument in the proof of
Lemma \ref{lem:parabolicsuper} that for each $Y \in L(\CA)$ of rank $2$ whose
$W$-orbit meets $L(\CA(W_X))$, $Y$ is not modular.
This leaves the $W$-orbit whose point stabilizer is of type $A_1^2$ 
to be considered.
Let $i = \sqrt{-1}$. 
The defining polynomial of $\CA$ is given by
\begin{align*}
Q(G_{29}) &:= 
a b c d 
(a - b)
(a - c)
(a - d)
(b - c)
(b - d)
(c - d)\\
&
(a + c)
(a + b)
(a + d)
(b + c)
(b + d)
(c + d)\\
&
(a - b + i c + id)
(a - b + i c -id)
(a - b -i c -i d)
(a - b -i c + i d)\\
& 
(a + b + ic + id)
(a + b -i c -i d)
(a + b -i c + id)
(a + b + ic -id)\\
&
(a -ib + i c + d)
(a -i b - c -i d)
(a -i b - c + i d)
(a -i b + i c - d)\\
&
(a -i b -i c + d)
(a -i b + c -i d)
(a - ib -ic - d)
(a + ib - c + id)\\
&
(a + i b - c  -i d)
(a + ib -ic + d)
(a + ib -i c - d)
(a + i b + c + i d)\\
&
(a + ib + ic + d)
(a + ib + ic - d)
(a -i b + c + id)
(a + ib + c -i d).
\end{align*}
An element in $L(\CA)$ with point stabilizer $A_1^2$ is given by 
$Y = H_{a-b+ic+id} \cap H_{a+ib-c-id}$. One checks that
$Y + (H_{a+ib-ic+d} \cap H_{b-d})= H_{a+(2i-1)b-ic-(i-2)d}\notin \CA$,
so $Y$ is not modular.

(iii). Let $W = G_{31}$.
Let $i = \sqrt{-1}$. 
The defining polynomial of $\CA = \CA(G_{31})$ is
given by 
\begin{align*}
Q(G_{31}) &  := 
a b c d
(a - b)
(a - c)
(a - d)
(b - c)
(b - d)
(c - d)\\
&
(a + b)
(a + c)
(a + d)
(b + c)
(b + d)
(c + d)\\
&
(a - ib)
(a - ic)
(a - id)
(b - ic)
(b - id)
(c - id)\\
&
(a + ib)
(a + ic)
(a + id)
(b + ic)
(b + id)
(c + id)\\
&
(a - b - c - d)
(a - b + c + d)
(a - b + c - d)
(a - b - c + d)\\
&
(a + b + c + d)
(a + b - c + d)
(a + b - c - d)
(a + b + c - d)\\
&
(a - b - ic - id)
(a - b + ic + id)
(a - b - ic + id)
(a - b + ic - id)\\
&
(a + b - ic - id)
(a + b + ic + id)
(a + b + ic - id)
(a + b - ic + id)\\
&
(a - ib - c - id)
(a - ib + c - id)
(a - ib - c + id)
(a - ib + c + id)\\
&
(a - ib + ic + d)
(a - ib + ic - d)
(a - ib - ic + d)
(a - ib - ic - d)\\
&
(a + ib + c - id)
(a + ib - c - id)
(a + ib - c + id)
(a + ib + c + id)\\
&
(a + ib - ic + d)
(a + ib - ic - d)
(a + ib + ic - d)
(a + ib + ic + d).
\end{align*} 
There are three $W$-orbits of elements of rank $2$ in $L(\CA)$, cf.\ 
\cite[Table C.12]{orlikterao:arrangements}. 
They correspond to the three conjugacy classes of
parabolic subgroups of $W$ of types 
$A_1^2$, $A_2$, and $G(4,2,2)$ with
respective representatives
\[
X_1  = H_a \cap H_{b-c}, \quad  
X_2  = H_{a+id} \cap H_{a+b-c-d}, \text{ and } \quad 
X_3  = H_{a} \cap H_{a+ib}.
\]
One checks that 
$X_1+X_2 = H_{2a+(1+i)b-(1+i)c} \notin \CA$, so
that  $X_1$ and $X_2$ are
not modular. 
Further, one calculates that
$X_3 + (H_{a-b-c-d}\cap H_{a-ic}\cap H_{a-b-c+d}) = 
H_{2a+(i-1)b} \notin \CA$,
so $X_3$ is not modular either.

This completes the proof of Lemma \ref{lem:super3}.
\end{proof}

\begin{lemma}
\label{lem:inductionsuper2}
If $W$ is a monomial group $G(r,r,\ell)$ for $r \ge 3$, $\ell \ge 5$, 
$G_{33}$, or $G_{34}$,
then there are no modular elements in $L(\CA(W))$ of rank $2$.
\end{lemma}

\begin{proof}
Let $\CA = \CA(W)$.
If $W = G(r,r,\ell)$, for $r \ge 3$, $\ell \ge 5$, $G_{33}$, and $G_{34}$, 
let $W_X$ be the parabolic subgroup $W_X = G(r,r,4)$, $W_X = W(D_4)$,
and $W_X = W(D_4)$, respectively.
By Lemmas \ref{lem:super} and \ref{lem:super3}, 
$L(\CA(W_X))$ does not admit modular elements of rank $2$.
{}From the information on the $W$-orbits in $L(\CA)$ given in 
\cite[\S 6, Table C.14, C.17]{orlikterao:arrangements},
we infer that with these choices, 
every $W$-orbit of elements in $L(\CA)$ of rank $2$ 
meets the sublattice $L(\CA(W_X))$.
The desired result follows from 
Lemma \ref{lem:parabolicsuper}.
\end{proof}

\begin{proof}[Proof of Theorem \ref{super-rank2}]

The forward implication is clear from Definition \ref{def:super}.

The reverse implication follows 
for general central arrangements of rank up to $3$
by Remark \ref{rem:2-arr} and Lemma \ref{lem:super-rank2}.
So assume that $\CA(W)$ is irreducible of rank at least $4$.
It follows readily from Theorem \ref{super} and Lemmas 
\ref{lem:super} 
through \ref{lem:inductionsuper2}
that if $\CA$ is not supersolvable, then
there are no modular elements of rank $2$ in $L(\CA)$.
\end{proof}

While Theorem \ref{super-rank2} asserts that 
an irreducible reflection arrangement is supersolvable precisely
when its lattice admits a modular element of rank $2$, 
this assertion is false for reducible reflection arrangements,
see Remark \ref{rem:reducible}.

The defining polynomials of the reflection arrangements
for $F_4$, $H_3$, $G_{25}$, $G_{26}$ and $G_{31}$ 
used in the proofs of 
Lemmas \ref{lem:super} and \ref{lem:super3} were obtained 
using the functionality for complex reflection groups 
provided by the \CHEVIE\ package in   \GAP\ 
(and some \GAP\ code by J.~Michel)
(see \cite{gap3} and \cite{chevie}).



\bigskip
\bigskip

\bibliographystyle{amsalpha}

\newcommand{\etalchar}[1]{$^{#1}$}
\providecommand{\bysame}{\leavevmode\hbox to3em{\hrulefill}\thinspace}
\providecommand{\MR}{\relax\ifhmode\unskip\space\fi MR }
\providecommand{\MRhref}[2]{%
  \href{http://www.ams.org/mathscinet-getitem?mr=#1}{#2} }
\providecommand{\href}[2]{#2}


\end{document}